\documentclass[12pt,reqno]{amsart}
\usepackage[cmtip,all]{xy}
\usepackage{graphicx}
\usepackage{amssymb}
\usepackage{mathrsfs}
\usepackage{amsfonts}
\usepackage{amsmath}
\usepackage[pagebackref]{hyperref}%
\newtheorem{theorem}{Theorem}[section]

\theoremstyle{definition}

\theoremstyle{remark}

\numberwithin{equation}{section}

\begin{document}

\title[Solving Schiffer's Problem Almost Surely]{Solving Schiffer's Problem in Inverse Scattering Theory Almost Surely}

\author{Hongyu Liu}
\address{Department of Mathematics, City University of Hong Kong, Hong Kong SAR, China}
\email{hongyu.liuip@gmail.com, hongyliu@cityu.edu.hk}

\thanks{}
\date{August 18, 2024 }


\maketitle

\begin{abstract}

In this short note, we present a probabilistic perspective on the Schiffer's problem in the inverse scattering theory, which asks whether one can uniquely determine the shape of an unknown obstacle by a single far-field measurement. It is a longstanding problem and has received considerable studies in the literature. We show that this conjecture holds true in more general settings in the probability sense. Our new perspective has important implications from the practical viewpoint and also points an interesting direction of research for broader inverse problems.

\end{abstract}

\section{Schiffer's problem in inverse scattering theory}

Let $k\in\mathbb{R}_+$ signify a wavenumber, and $\Omega\subset\mathbb{R}^n$, $n\geq 2$, denote an impenetrable obstacle. It is assumed that $\Omega$ is a bounded Lipschitz domain with a connected complement $\mathbb{R}^n\backslash\overline{\Omega}$. Let $u^i$ be an entire solution to the following Helmholtz equation:
\begin{equation}\label{eq:h1}
(\Delta+k^2)u^i=0\quad\mbox{in}\ \ \mathbb{R}^n. 
\end{equation}
We consider the following time-harmonic scattering problem:
\begin{equation}\label{eq:h2}
\begin{cases}
(\Delta+k^2) u=0,\quad & \mbox{in}\ \mathbb{R}^n\backslash\overline{\Omega},\medskip\\
\mathcal{B}(u)=0,\quad & \mbox{on}\ \partial\Omega,\medskip\\
u^s(x):=u(x)-u^i(x),\quad & x\in\mathbb{R}^n\backslash\overline{\Omega},\medskip\\
\lim_{|x|\rightarrow\infty} |x|^{(n-1)/2}\left(\frac{\partial u^s}{\partial |x|}-\mathrm{i}k u^s\right)=0,
\end{cases}
\end{equation}
where $\mathrm{i}:=\sqrt{-1}$, and $\mathcal{B}(u)=u$ or $\partial u/\partial\nu$ or $\partial u/\partial\nu+\lambda u$. Here, $\nu\in\mathbb{S}^{n-1}$ denotes the exterior unit normal to $\partial\Omega$, and $\lambda\in L^\infty(\partial\Omega)$ with $\Im\lambda>0$ signifies a boundary impedance parameter. It is known that there exists a unique solution $u\in H_{loc}^1(\mathbb{R}^n\backslash\overline{\Omega})$ which possesses the following asymptotic expansion (cf. \cite{CK,LL}):
\begin{equation}\label{eq:h3}
u^s(x)=\frac{e^{\mathrm{i}k|x|}}{|x|^{(n-1)/2}} u_\infty(\hat{x})+\mathcal{O}\left(\frac{1}{|x|^{(n+1)/2}}\right),\quad \hat x:=\frac{x}{|x|}\in\mathbb{S}^{n-1}, 
\end{equation}
where $u_\infty$ is known as the far-field pattern.

In the physical setup, one sends a probing wave $u^i$ to interrogate the obstacle $(\Omega, \mathcal{B})$ and generate the total wave field $u$. The scattered wave field $u^s$ is then generated by the linear superposition given in the third equation of \eqref{eq:h2}. In general, one takes the incident wave $u^i$ to be the time-harmonic plane wave of the form:
\begin{equation}\label{eq:p1}
u^i(x)=e^{\mathrm{i}kx\cdot d},\quad d\in\mathbb{S}^{n-1},
\end{equation}
where $d$ signifies the incident direction. The scattering information is completely encoded into the far-field pattern, which we write as $u_\infty(\hat x, u^i, (\Omega, \mathcal{B}))=u^i(\hat x, d, k, (\Omega, \mathcal{B}))$. Hence, an inverse scattering problem of practical importance is to recover the obstacle by measurement of the far-field pattern:
\begin{equation}\label{eq:h4}
u_\infty\longrightarrow (\Omega, \mathcal{B}). 
\end{equation}
The unique identifiability issue for this inverse problem asks whether the correspondence between $u_\infty$ and $(\Omega, \mathcal{B})$ is one-to-one. There is a widespread belief that one can establish the above unique correspondence by using a single far-field pattern, namely $u_\infty(\hat x)$ with $\hat x\in\mathbb{S}^{n-1}$ generated from a fixed $u^i$. By cardinality counting, it can be easily seen that the problem is formally determined in this case. However, it remains a longstanding problems in the literature, and is known as the Schiffer's problem in the inverse scattering theory. In fact, M. Schiffer made a pioneering contribution by showing that one can establish the unique identifiability result by using infinitely many measurements \cite{LP}. The other known results by a single far-field measurement require the obstacle from an a-priori shape class, say e.g. balls, or polytopes. Instead of discussing the long list of existing literature, we refer to \cite{CK,CK2,DL,L1} for more comprehensive reviews and surveys. 

In this short note, we show that the Schiffer's problem can be solved in more general setups almost surely from a certain probability sense. Our argument is basically based on the spectral idea of M. Schiffer as well as the analytic dependence of $u_\infty$ on its arguments. Nevertheless, we still think it is worth presenting the results. This novel perspective brings new insights of practical importance. In fact, the inverse problem \eqref{eq:h4} can be recast as the following operator equation:
\begin{equation}\label{eq:h5}
\mathcal{F}((\Omega, \mathcal{B}))=u_\infty(\hat x, u^i), 
\end{equation}
where the nonlinear operator $\mathcal{F}$ is defined by the forward scattering system \eqref{eq:h2}. Our study basically indicates that one can pick up a fixed incident field $u^i$ to generate the far-field pattern for solving \eqref{eq:h5}, then the chance of failing to recover $(\Omega,\mathcal{B})$ is zero. Furthermore, from an algorithmic point of view, the equation \eqref{eq:h5} is usually solved in an iterative way, and if restarting process is needed, one can choose to use one more far-field pattern generated by a different $u^i$. In doing this, the chance to succeed is almost surely guaranteed. This also explains why in many existing numerical results in the literature, one can always reconstruct the obstacle by using a few (not too many) far-field measurements.

\section{The Shiffer's problem is solved almost surely}

In order to ease the exposition, we introduce several simplifications. It is emphasised that those simplifications are not necessary from the technical point of view, and all of our results hold in the general setting. First, we assume that $\Omega$ is simply connected in the Lipschitz class. Second, we always take the incident wave to be the plane wave \eqref{eq:p1}. Third, more than often, we take $\mathcal{B}(u)=u$ as the representative case for our discussion, which corresponds to $\Omega$ being a so-called sound-soft obstacle. Nevertheless, if necessary, we also discuss the needed extensions to the other cases with $\mathcal{B}(u)=\partial_\nu u$ (sound-hard obstacle), or $\mathcal{B}(u)=\partial_\nu u+\lambda u$ (impedance obstacle). Finally, we shall mainly consider the recovery of the shape of the obstacle, namely $\Omega$, but not its physical type, namely $\mathcal{B}$. In fact, as soon as $\Omega$ can be determined, $\mathcal{B}$ can be readily determined by the classical UCP property; see e.g. \cite{DL}. 

We first discuss the original idea of M. Schiffer who proved that a sound-soft obstacle $\Omega$ can be uniquely determined by $u_\infty(\hat x, d, k)$ with: (i). $k\in I$ with $I$ being an open interval in $\mathbb{R}_+$ and $d\in\mathbb{R}_+$ is fixed; or (ii). $k\in\mathbb{R}_+$ is fixed and $d\in\mathcal{D}$ with $\mathcal{D}$ being an infinite set in $\mathbb{S}^{n-1}$. We refer to \cite{CK} for the detailed proof. Here, for our subsequent need, we discuss the key observation. Let $\Omega$ and $\Omega'$ be two sound-soft obstacles and assume that they produce the same far-field pattern, namely $u_\infty(\hat x, \Omega)=u_\infty(\hat x, \Omega')$ for all $\hat x\in\mathbb{S}^{n-1}$. By Rellich's Theorem \cite{CK}, one has that $u=u'$ in the unbounded connected component of $\mathbb{R}^n\backslash\overline{(\Omega\cup\Omega')}$ (written as $\Omega_c$ in what follows), where $u$ and $u'$ are respectively the total fields correspond to $\Omega$ and $\Omega'$. If $\Omega\neq \Omega'$, one can assume without of generality that there exists a connected component of $\mathbb{R}^n\backslash\Omega_c$, written as $\Omega^*$, such that $\Omega^*\subset\mathbb{R}^n\backslash\overline{\Omega}$. It is clear that
\begin{equation}\label{eq:s1}
-\Delta u=k^2 u\quad\mbox{in}\ \ \Omega^*, \quad u|_{\partial\Omega^*}=0. 
\end{equation}
That is, $\lambda=k^2$ is a Dirichlet Laplacian eigenvalue to $\Omega^*$ and $u|_{\Omega}$ is the corresponding eigenfunction. Hence, to prove the assertion in (i), it is sufficient to note the discreteness of the set of Dirichlet Laplacian eigenvalues; and to prove the assertion in (ii), it is sufficient to use the fact that for a fixed eigenvalue $k^2$, the corresponding eigen-space is finite dimensional. 

We recall the Borel probability measure $\mu$ on any bounded open set $\mathcal{M}\subset\mathbb{R}^n$. For an event set $E\subset\mathcal{M}$, the probability is defined in the usual way as $\mu(E)/\mu(\mathcal{M})$. We have that

\begin{theorem}\label{thm:1}
Let $\Omega$ and $\Omega'$ be two sound-soft obstacles. If $\Omega\neq\Omega'$, then it holds almost surely (a.s.) that
\begin{equation}\label{eq:s1}
u_\infty(\Omega, e^{\mathrm{i}kx\cdot d})\equiv\hspace*{-4mm}\backslash\ u_\infty(\Omega', e^{\mathrm{i}kx\cdot d}),
\end{equation}
for a given $e^{\mathrm{i}kx\cdot d}$ with $k$ being randomly chosen from $(\underline{k}, \overline{k})$ and $d$ being randomly chosen from $\mathbb{S}^{n-1}$, respectively, where $\underline{k}$ and $\overline{k}$ are two given finite numbers in $\mathbb{R}_+$. 
\end{theorem}

\begin{proof}
From Schiffer's argument, it is easily seen that the probability measures for the sets of $k$ and $d$ such that $u_\infty(\Omega)=u_\infty(\Omega')$ are both zero. 
\end{proof}

Theorem~\ref{thm:1} is a probabilistic interpretation of Schiffer's result. However, it cannot be extended to the case with sound-hard or impedance obstacles. Indeed, (e.g) in the sound-hard case, the spectral property of the Neumann Laplacian in $\Omega^*$ critically depends the regularity of $\Omega^*$, but which could be very ``irregular", making the situation radically much more complicated. Nevertheless, from the probability perspective, we can have

\begin{theorem}\label{thm:2}
Let $\Omega$ and $\Omega'$ be two sound-hard/impedance obstacles. If $u_\infty(\Omega, e^{\mathrm{i}kx\cdot d})\equiv u_\infty(\Omega, e^{\mathrm{i}kx\cdot d})$ for a given $e^{\mathrm{i}kx\cdot d}$ with $k\in (\underline{k}, \overline{k})$ and $d\in\mathbb{S}^{n-1}$ being randomly chosen respectively, then it holds a.s. that
\begin{equation}\label{eq:s2}
d_{\mathcal{H}}(\Omega, \Omega')\leq \mathrm{diam}(\Omega)+\mathrm{diam}(\Omega'), 
\end{equation}
where $d_\mathcal{H}$ denotes the Hausdorff distance and $\mathrm{diam}$ denotes the diameter. 
\end{theorem}

\begin{proof}
Let us only consider the sound-hard case. If $\overline{\Omega}\cap\overline{\Omega'}=\emptyset$, then $k^2$ is a Neumann Laplacian eigenvalue in $\Omega$ with the eigenfunction being $u'$, and it is also a Neumann Laplacian eigenvalue in $\Omega'$ with the eigenfunction being $u$. Following the spirit of Schiffer's argument, one can show that: (i). with a fixed $d$, the possible $k$'s form a discrete set; (ii). with a fixed $k$, the possible $d$'s are finitely many. Hence, the probability measure for the set of those ``bad" $k$'s and $d$'s are zero. Therefore, it holds a.s. that $\Omega\cap\Omega'\neq \emptyset$, which readily implies \eqref{eq:s2}. 
\end{proof}

Theorem~\ref{thm:2} has some interesting physical implication. In fact, if we have some a-priori knowledge about the size of the obstacle, say e.g. $\mathrm{diam}(\Omega)\leq L$, then Theorem~\ref{thm:2} indicates that by a single far-field measurement, one can recover the obstacle with a high probability not far-away from its true shape (in fact, in a vicinity of distance $L$ from the true shape). 

Next, we consider a more practical situation, where one can take the radar/sonar imaging as a practical motivation. Let $\mathcal{A}=\{\Omega_j\}_{j=0}^N$ be an a-priori set of base-shapes. The target obstacle $\Omega$ is obtained from translating or rotating those base-shapes:
\begin{equation}\label{eq:r1}
\Omega=z_0+\mathcal{R}_U(\Omega_j), \quad z_0\in\mathbb{R}^n,
\end{equation}
where $\mathcal{R}_U(\Omega_j):=U\Omega_j=\{Ux; x\in \Omega_j\}$ with $U\in SO(n)$ being a rotation matrix. Taking the radar industry as an example, one knows in advance all the possible types of the target aircrafts and hence all the possible base-shapes. Next, in order to ease the exposition, we assume that they are sound-hard obstacles. Moreover, it is unobjectionable to assume that 
\begin{equation}\label{eq:r2}
u_\infty(\Omega_j, e^{\mathrm{i}kx\cdot d})\equiv\hspace*{-4mm}\backslash\ u_\infty(\Omega_l, e^{\mathrm{i}kx\cdot d}), \quad \Omega_j, \Omega_l\in\mathcal{A},\ \ j\neq l,
\end{equation}
for any given $e^{\mathrm{i}k x\cdot d}$ with $k\in (\underline{k}, \overline{k})$ and $d\in\mathbb{S}^{n-1}$. In fact, if \eqref{eq:r2} does not hold true, one readily has a counter example to the Schiffer's problem. Finally, we assume that the location of $\Omega$, namely $z_0$, can be uniquely determined a priori. In fact, the determination of the location can be conducted in a rather simple and stable manner in the inverse scattering theory; see e.g. \cite{LLZ}. 

We can show that

\begin{theorem}\label{thm:3}
Let $\Omega$ be descried as above. The $\Omega$ can be uniquely determined a.s. by a single far-field measurement associated with $e^{\mathrm{i}k x\cdot d}$ with $k\in (\underline{k}, \overline{k})$ being fixed and $d\in\mathbb{S}^{n-1}$ being randomly chosen.
\end{theorem}

\begin{proof}
Without loss of generality, we can assume that 
\begin{equation}\label{eq:s1}
\Omega=z_0+\mathcal{R}_{U_0}(\Omega_0),
\end{equation}
where $z_0\in\mathbb{R}^n$ and $U_0\in SO(n)$. In what follows, we let $U\in SO(n)$ be parametrised as $U(\theta)$ with $\theta\in \Sigma$, where $\Sigma$ is a compact set of $\mathbb{R}^{n-1}$. Suppose that there exists $\Omega'=z+\mathcal{R}_{U}(\Omega_j)$, $j\neq 0$, such that
\begin{equation}\label{eq:s2}
u_\infty(\Omega, e^{\mathrm{i}kx\cdot d})=u_\infty(\Omega', e^{\mathrm{i}kx\cdot d}),
\end{equation} 
for a given $e^{\mathrm{i}kx\cdot d}$. Since the location can be a priori determined, we have $z=z_0$. By using Lemma~2.1 in \cite{LLZ}, we have 
\begin{equation}\label{eq:s3}
u_\infty(\mathcal{R}_{U_0}(\Omega_0), e^{\mathrm{i}kx\cdot d})=e^{\mathrm{i}k(\hat x-d)\cdot z_0} u_\infty(z_0+\mathcal{R}_{U_0}(\Omega_0), e^{\mathrm{i}kx\cdot d}),
\end{equation}
and 
\begin{equation}\label{eq:s4}
u_\infty(\mathcal{R}_{U}(\Omega_j), e^{\mathrm{i}kx\cdot d})=e^{\mathrm{i}k(\hat x-d)\cdot z_0} u_\infty(z_0+\mathcal{R}_{U}(\Omega_j), e^{\mathrm{i}kx\cdot d}). 
\end{equation}
By \eqref{eq:s2}--\eqref{eq:s4}, it further gives that 
\begin{equation}\label{eq:s5}
u_\infty(\mathcal{R}_{U_0}(\Omega_0), e^{\mathrm{i}kx\cdot d})=u_\infty(\mathcal{R}_{U}(\Omega_j), e^{\mathrm{i}kx\cdot d}). 
\end{equation}
We claim that there cannot exist a set $\Gamma\subset\Sigma$ with $\mu(\Gamma)>0$ such that \eqref{eq:s5} holds for $U(\theta)\in SO(n)$, $\theta\in\Gamma$. In fact, we note that $u_\infty$ is (real) analytic with respect to $\theta$, if the above claim does not hold, one has by analytic continuation that \eqref{eq:s5} holds for any $U\in SO(n)$, and in particular, 
\begin{equation}\label{eq:s6}
u_\infty(\mathcal{R}_{U_0}(\Omega_0), e^{\mathrm{i}kx\cdot d})=u_\infty(\mathcal{R}_{U_0}(\Omega_j), e^{\mathrm{i}kx\cdot d}). 
\end{equation}
By Lemma~3.1 in \cite{LLZ}, we have
\begin{equation}\label{eq:s7}
u_\infty(\hat x, \mathcal{R}_{U_0}(\Omega_0), e^{\mathrm{i}kx\cdot d})=u_\infty(U_0^T\hat x, \Omega_0, e^{\mathrm{i}kx\cdot (U_0^T d)}),
\end{equation}
and
\begin{equation}\label{eq:s8}
u_\infty(\hat x, \mathcal{R}_{U_0}(\Omega_j), e^{\mathrm{i}kx\cdot d})=u_\infty(U_0^T\hat x, \Omega_j, e^{\mathrm{i}kx\cdot (U_0^T d)}),
\end{equation}
which in combination with \eqref{eq:s6} readily gives that
\begin{equation}\label{eq:s9}
u_\infty(\Omega_0, e^{\mathrm{k}x\cdot d'})=u_\infty(\Omega_j, e^{\mathrm{i}kx\cdot d'}), \quad d':=U_0^T d. 
\end{equation}
This clearly contradicts to \eqref{eq:r2}. 

On the other hand, one can also show that there cannot exists a set $\mathcal{D}\subset\mathbb{S}^{n-1}$ with positive measure such that \eqref{eq:s2} holds for all $d\in\mathcal{D}$, since otherwise by analytic continuation again, one has that \eqref{eq:s2} holds for all $d\in\mathbb{S}^{n-1}$. Then by using the known unique identifiability result (cf. \cite{}), one has $\Omega=\Omega'$. 

In conclusion, if \eqref{eq:s2} holds, it is a.s. that $\Omega=\Omega'$. The proof is complete.

\end{proof}

 \section{Outlook and viewpoint}
 
 In this short note, we consider the longstanding Schiffer's problem in the inverse scattering theory, arguing that it can actually be solved from a certain probability sense. Theorems~\ref{thm:1} and \ref{thm:2} basically indicate that by by randomly choosing an incident wave to generate a single far-field measurement, one can distinguish two obstacles almost surely. Theorem~\ref{thm:3} states that if the obstacle is from an a-priori class with a ``finite-dimensional" generating set, then by randomly choosing an incident wave to generate a single far-field measurement, one can uniquely determine the underlying obstacle almost surely. Here, the ``finite dimensionalbity" is not understood in the usual linear sense. It actually means an unknown configuration with a finite number of parameters. Clearly, the theorem can be easily extended to the case with countably many generating base shapes. This new perspective brings interesting implications to practical applications, especially from the algorithmic point of view. Clearly, this viewpoint can be extended to the other scattering problems, say e.g. electromagnetic or elastic scattering \cite{DL}, or even broader inverse problems. In fact, for many inverse problems, one usually establishes conditional unique identifiability results. Those so-called ``conditions" define the a-priori classes of the unknown target objects. This is totally understandable since it is widely accepted that for inverse problems, no mathematical trickeries can remedy the lack of a-priori information. But here we made a crucial observation is that the measurement dataset usually depends on the a-priori class analytically. Hence, if the a-priori class possesses a countable/separable generating set, one should be able to establish more practical unique identifiability results in the probability sense. Nevertheless, it is pointed out that inverse problems are generically ill-conditioned, and it is of high practical and theoretical interest to establish the stability results in the probability sense. Let us still take the Schiffer's problem as an example. In fact, though for a randomly chosen $k$ or $d$, the failure probability of uniquely determining the obstacle $\Omega$ is zero, the recovery can be nearly failed locally around those potential ``failing" $k$ or $d$, whose probability can be nonzero. Clearly, such a ``stable success" probability depends on the a-priori class as well as the underlying physics of the inverse problem.  
 
Finally, we would like to emphasise that deriving a deterministic answer to the Schiffer's problem still remains one of the crown jewels in inverse scattering theory.

\section*{Acknowledgement}
The research was supported by NSFC/RGC Joint Research Scheme, N CityU101/21, ANR/RGC
Joint Research Scheme, A-CityU203/19, and the Hong Kong RGC General Research Funds (projects 11311122, 11304224 and 11300821).

\bibliographystyle{amsalpha}

\end{document}